\documentclass[12pt,leqno,fleqn]{amsart}  
\usepackage{amsmath,amstext,amsthm,amssymb,amsxtra}
\usepackage[top=1.5in, bottom=1.5in, left=1.25in, right=1.25in]{geometry}
\usepackage{txfonts} % pxfonts txfonts 
\usepackage[T1]{fontenc}
\usepackage{lmodern}

 \usepackage{euler}   % better than the option below

\usepackage{mathtools}
\mathtoolsset{showonlyrefs,showmanualtags} 

\usepackage{hyperref} %,hypertexnames=false,colorlinks,[pagebackref]
\hypersetup{
    colorlinks=true,       % false: boxed links; true: colored links
    linkcolor=blue,          % color of internal links
    citecolor=magenta,        % color of links to bibliography
    filecolor=magenta,      % color of file links
    urlcolor=cyan           % color of external links
}

\usepackage[shortalphabetic,backrefs]{amsrefs}  
\theoremstyle{plain} % definition 
\newtheorem{lemma}[equation]{Lemma} 
 
\newtheorem{theorem}[equation]{Theorem}

\theoremstyle{definition}

\theoremstyle{remark}

\newtheorem*{ack}{Acknowledgment}

\numberwithin{equation}{section}

\title {An Elementary Proof of the $ A_2$ Bound}
\author[M. T. Lacey]{Michael T. Lacey}   %  can use \and  

\address{ School of Mathematics, Georgia Institute of Technology, Atlanta GA 30332, USA}
\email {lacey@math.gatech.edu}
\thanks{Research supported in part by grant NSF-DMS 1265570. }

\begin{document}

\maketitle 

\begin{abstract}
A martingale transform $ T$,  applied to an integrable locally supported function $ f$, is 
pointwise dominated by  a positive sparse operator applied to $ \lvert  f\rvert $, 
the choice of sparse operator being a function of $ T$ and $ f$.  
As a corollary, one derives the sharp $ A_p$ bounds for martingale transforms, recently proved by 
Thiele-Treil-Volberg, as well as a number of  new sharp weighted inequalities for martingale transforms. 
The (very easy) method of proof (a) only depends upon the weak-$L^1$ norm of maximal truncations of martingale 
transforms, (b)  applies in the  vector valued setting, and (c) has an extension to the 
continuous case, giving a new elementary proof of the $ A_2$ bounds in that setting. 
\end{abstract}

%%%%%%%%%%%%%%%%%%%%%%%%%%%%%% SECTION  SECTION SECTION
%%%%%%%%%%%%%%%%%%%%%%%%%%%%%% SECTION  SECTION SECTION 
\section{Introduction} %\label{s:}

Our subject is weighted inequalities in Harmonic Analysis,  which started in the 1970's, 
and has been quite active in recent years.  A succinct history begins in 1973, 
when Hunt-Muckenhoupt-Wheeden \cite{MR0312139}
proved that the Hilbert transform $ H$ was bounded on $ L ^2 (w)$, for non-negative weight $ w$, 
if and only if the weight satisfied the Muckenhoupt  $ A_2$ condition \cite{MR0293384}.  One could have 
asked, even then, what the sharp dependence on the norm of $ H$ is, in terms of the $ A_2$ characteristic 
of the weight.  This was not resolved until 2007, by Petermichl \cite{MR2354322}, showing that norm of the Hilbert 
transform is linear in the $ A_2$ characteristic.  Five years later, 
among competing approaches,  Hyt\"onen \cite{MR2912709} established the $ A_2$ Theorem: For any 
Calder\'on-Zygmund operator $ T$, and any $ A_2$ weight, the norm estimate is linear in $ A_2$ characteristic.  
Hyt\"onen's original method depended upon the \emph{Hyt\"onen Representation Theorem}, \cite[Thm. 4.2]{MR2912709}, a novel  
expansion of $ T$ into a rapidly convergent series of \emph{discrete dyadic operators.}  
Subsequently, Lerner \cite{MR3085756} gave an alternate proof, cleverly  exploiting  his \emph{local mean oscillation inequality \cite{MR2721744}}  to prove a remarkable statement: 
The norm of Calder\'on-Zygmund operators on a Banach lattice is dominated by the operator norm of  a
class of very simple \emph{positive sparse operators} $ S \lvert  f\rvert $.  See \cite[(1.2)] {MR3085756}. 
The $ A_2$ bound is very easy  to establish for the sparse operators.   For more details on the history, see 
\cites{MR3204859,MR3085756}.  

The purpose of this note is to establish the \emph{pointwise control of $ Tf$}  by a sparse operator in a direct elementary fashion. 
 For each $ T$ and suitable $ f$, there is a choice of sparse operator $ \mathsf S$ 
so that $ \lvert  Tf\rvert \lesssim \mathsf S \lvert  f\rvert  $. 
Moreover, the recursive proof only requires elementary facts about  dyadic grids, and the weak integrability of the maximal 
truncations of $ T$.  (That is, we do not need Lerner's  mean oscillation inequality.)
Before this argument, pointwise bounds were independently established by \cites{14094351,150805639}. 
This argument, in the continuous case, requires a bit more on the operators than our result.

The argument is especially transparent in the setting of martingale transforms \S\ref{s:mart}. 
 The main result, Theorem~\ref{t:lerner},  proved in \S\ref{s:proof}, is new, and cannot be proved by 
 Lerner's  mean oscillation inequality.  
As  corollaries, we deduce the $ A_2$ bound for martingale transforms recently established by Thiele-Treil-Volberg \cite{TTV}, who employ more sophisticated techniques.  
A number of new, sharp weighted inequalities for martingales also follow as corollaries. 
The argument can be formulated in the Euclidean setting, giving a new self-contained proof 
of the $ A_2$ Theorem in that setting. 
Remarks and complements conclude the  paper. 

\begin{ack}
Guillermo Rey kindly pointed out some relevant references, as did the referee, who also 
pointed out some additional obscurities.  
\end{ack}

%%%%%%%%%%%%%%%%%%%%%%%%%%%%%% SECTION  SECTION SECTION
%%%%%%%%%%%%%%%%%%%%%%%%%%%%%% SECTION  SECTION SECTION 
\section{Martingale Transforms} \label{s:mart}  

We introduce standard notation required for a discrete time martingale. 
Let $ (\Omega , \mathcal A, \mu )$ be a $ \sigma $-finite measure space, and let $ \{\mathcal Q_n \;:\; n\in \mathbb Z \}$ 
be  collections of measurable sets $\mathcal Q_n \subset \mathcal A $ such that 
%%  ENUMERATE
\begin{enumerate}
\item   Each $ \mathcal Q_n$ is a partition of $ \Omega $, and $ 0 < \mu (Q) < \infty $ for all $ Q\in \mathcal Q_n$. 
\item  For each $ n$, and $ Q\in \mathcal Q_n$, the set $ Q$ is the union of sets $ Q'\in \mathcal Q _{n+1}$. 
(And $ Q'$ is allowed to equal $ Q$.) 
\item  The collection $ \mathcal Q :=  \bigcup _{n \in \mathbb Z }\mathcal Q_n $ generates $ \mathcal A$.  
\end{enumerate}
%% ENUMERATE
The set $ \mathcal Q$ is a tree  with respect to inclusion. (And the branching of the tree is unbounded in general.) 
Set $ Q ^{a}$ to be the minimal element $ P\in \mathcal Q$ which strictly contains $ Q$.  
In words,  $ Q ^{a}$ is the \emph{parent} of $ Q$.  Denote by $ \textup{ch} (P)$ those $ Q\in \mathcal Q$ 
such that $ Q ^{a}=P$, that is the \emph{children} of $ P$. 

The conditional expectation associated with $ \mathcal Q_n$ is 
\begin{equation}\label{e:EQ}
\mathbb E (f \;\vert\; \mathcal Q_n) := \sum_{Q\in \mathcal Q_n}  \mu (Q) ^{-1} \int _{Q} f \; d \mu \cdot \mathbf 1_{Q}. 
\end{equation}
Below, we will frequently use the abbreviation $ \langle  f \rangle_Q := \mu (Q) ^{-1} \int _{Q} f \; d \mu$. 
The \emph{martingale difference} associated with $ P \in \mathcal Q$ is 
\begin{equation}\label{e:dQ}
\Delta _P f :=  \sum_{\substack{Q\in  \textup{ch} (P)} }  \mathbf 1_{Q}  ( \langle f \rangle_{Q} - \langle f \rangle_P). 
\end{equation}
One should note that  $ P\in \mathcal Q_n$ is allowed to be in $ \mathcal Q _{n+1}$, 
so that the difference above is equal to  $ (\mathbb E  (f \mathbf 1_{P} \;\vert\; \mathcal Q _{n+1}) - \langle f \rangle_P) \mathbf 1_{P}$ only if $ P\in \mathcal Q_n$, but not $ \mathcal Q _{n+1}$.  

There are two kinds of operators we consider. The first are  \emph{martingale transforms}, namely operators of the form 
\begin{equation} \label{e:MT} 
T f = \sum_{P\in \mathcal Q} \epsilon _P \Delta _P f , \qquad \lvert  \epsilon _P\rvert\leq 1.  
\end{equation}
By the joint orthogonality of the martingale differences, $ T$ is an $ L ^2 $-bounded operator. 

And the second category are sparse operators.  An operator $ \mathsf S$ is \emph{sparse} if  
\begin{equation} \label{e:Sparse}
\mathsf S f = \sum_{P\in \mathcal S} \langle f \rangle_P \mathbf 1_{P} 
\end{equation}
where the collection $ \mathcal S$ satisfies this \emph{sparseness condition}:  
For each $ P\in \mathcal S$,  
\begin{equation}\label{e:sparse}
 \sum_{Q\in \textup{ch} _{\mathcal S} (P) }  \mu (Q) \leq \tfrac 12 \mu (P). 
\end{equation}
The sum on the left is restricted to the $ \mathcal S$-children of $ P$: The maximal elements of $ \mathcal S$ that 
are strictly contained in $ P$.  
It is well-known that both classes of operators are bounded in $ L ^{p}$, with constants uniform over the class of operators.

One of the main results is this extension of the Lerner inequality \cite{MR3085756} to the 
martingale setting, with a very simple proof, see \S\ref{s:proof}.

%%%%%%%%%%%%%%%%%%%%%%%%%%%%%% LEMMA LEMMA LEMMA
\begin{theorem}\label{t:lerner}  There is a constant $ C>0$ 
so that for all   functions  $ f\in L ^{1} (\mu )$ supported on $ Q_0 \in \mathcal Q$, 
and  martingale transforms   $ T$,  there is a sparse operator $ \mathsf S = \mathsf S _{T, f}$ so that 
\begin{equation}\label{e:lerner}
\mathbf 1_{Q_0}\lvert  T f \rvert \leq C \cdot  \mathsf S \lvert  f\rvert.  
\end{equation}
The same inequality holds for the maximal truncations, $ T _\sharp$, as defined in \eqref{e:sharp}. 

\end{theorem}
%%%%%%%%%%%%%%%%%%%%%%%%%%%%%% LEMMA LEMMA LEMMA

From this Theorem, one can deduce a range of sharp inequalities for martingale transforms by 
simply repeating the proofs for sparse operators on Euclidean space, using for instance the arguments in
\cite[\S2]{MR3127380}.

Consider two weights $ w, \sigma $. For pair, we have the joint $ A_p$ condition, $ 1< p < \infty $,  given by 
\begin{equation}\label{e:jointA2}
[\sigma, w ] _{A_p} := \sup _{Q\in \mathcal Q } \langle \sigma  \rangle_Q  ^{p-1}\langle w \rangle_Q  
= \sup _{n}\bigl[ \lVert \mathbb E (\sigma \;\vert\; \mathcal Q_n)  \bigr] ^{p-1}\mathbb E (w \;\vert\; \mathcal Q_n)\rVert _{\infty } . 
\end{equation}
If $ w $ is non-negative $ \mu $-almost everywhere, and $  w ^{1/(1-p)}$ is also a weight  (that is, $ w ^{1/(1-p)}$ is  locally integrable), 
then $ [w ^{1/(1-p)} ,w] _{A_p} = [w] _{A_p}$ is the $ A_p$ characteristic of $ w$.  

This is the $ A_2$ Theorem in the martingale setting, as proved by \cite{TTV}.  
After the domination by sparse operators, one can follow the short and self-contained
proof of \cite{MR3000426}*{Thm 3.1}.

%%%%%%%%%%%%%%%%%%%%%%%%%%%%%% THEOREM THEOREM THEOREM
\begin{theorem}\label{t:A2}  For any martingale transform or sparse operators $ T$, 
 for any $ 1< p < \infty $,  and $ A_p $ weights $ w$, 
\begin{equation}
\lVert T   \;:\; L ^p (w ) \mapsto L ^p (w)\rVert   \lesssim [w] _{A_p} ^{ \min\{ 1, 1/(p-1)\}} .   
\end{equation}
\end{theorem}
%%%%%%%%%%%%%%%%%%%%%%%%%%%%%% THEOREM THEOREM THEOREM

One draws the corollary  that the martingale differences  $ \Delta _Q $ are unconditionally convergent 
in $ L ^{p} (w)$, for all $ 1< p < \infty $, and all $ A_p$ weights $ w$, which  fact has been known 
since the 1970's.  We return to  this topic in \S\ref{s:complements}.

Many finer results are also corollaries, and the reader can use \cite[\S2]{MR3127380}, and  \cites{MR3129101,13103507,MR3127385,MR3085756,14080385} as guides to these results. 

One further application is to certain martingale paraproduct operators.  The Theorem below follows from an easy 
modification of  the proof that follows,  but we suppress the details.  
Define an operator $ \Pi  f := \sum_{Q \in \mathcal Q} \langle f \rangle_Q \cdot b_Q$, where 
$ \{b_Q \;:\; Q\in \mathcal Q\}$ are a sequence of functions which satisfy $ \Delta _Q b_Q = b_Q$, and this 
normalization condition. 
\begin{equation*}
\sup _{P} \mu (P) ^{-1} \sum_{Q \;:\; Q\subset P} \lVert b_Q\rVert _{\infty } ^2 \mu (  Q) \leq 1.  
\end{equation*}
It is an easy consequence of the Carleson embedding theorem that $ \Pi  $ is an $ L ^2 $-bounded operator.  
They are considered in \cite{14080385}*{Thm 2.5}, and they too can be dominated by a sparse operator. 

%%%%%%%%%%%%%%%%%%%%%%%%%%%%%% THEOREM THEOREM THEOREM
\begin{theorem}\label{t:para}  Let $ f \in L ^{1} (\Omega ) $ be supported on $ Q_0\in \mathcal Q$. 
Then, there is a sparse operator $ \mathsf S$ such that 
$ \mathbf 1_{Q_0} \lvert  \Pi f \rvert \lesssim \mathsf S  f $.  
\end{theorem}
%%%%%%%%%%%%%%%%%%%%%%%%%%%%%% THEOREM THEOREM THEOREM

A number of weighted inequalities for these operators is an immediate consequence, including both of the  main results of  \cite{14080385}. 

%%%%%%%%%%%%%%%%%%%%%%%%%%%%%% SECTION  SECTION SECTION
%%%%%%%%%%%%%%%%%%%%%%%%%%%%%% SECTION  SECTION SECTION 
\section{Domination by Sparse Operators: Martingales} \label{s:proof}

We give the proof of Theorem~\ref{t:lerner}, with the  basic fact being the 
weak-$ L ^{1}$ inequality for maximal truncations of martingale  transforms.

%%%%%%%%%%%%%%%%%%%%%%%%%%%%%% THEOREM THEOREM THEOREM
\begin{theorem}\label{t:1}[Burkholder \cite[Thm 6]{MR0208647}] For any martingale transform $ T$ we have 
\begin{equation}\label{e:1}
\sup _{\lambda >0} \lambda \mu ( T _\sharp f > \lambda ) \lesssim   \lVert f\rVert _{L ^{1} (\mu )}, 
\end{equation}
where $ T _\sharp f$ is the maximal truncation, given by 
\begin{equation} \label{e:sharp}
T _\sharp f :=  \sup _{Q'} \Bigl\lvert \sum_{\substack{Q \in \mathcal P \\ Q\supset Q'}} \epsilon _Q \Delta _Q f \Bigr\rvert. 
\end{equation}

\end{theorem}
%%%%%%%%%%%%%%%%%%%%%%%%%%%%%% THEOREM THEOREM THEOREM

%%%%%%%%%%%%%%%%%%%%%%%%%%%%%% PROOF PROOF PROOF
\begin{proof}[Proof of Theorem~\ref{t:lerner}] 
We concentrate on the case of  $ T$ being a martingale transform as in \eqref{e:MT}, with the changes to accommodate the maximal truncations being  easy to provide.

Apply the weak $ L ^{1}$ inequality for the maximal function $ M f $, 
and the maximal truncations of the martingale transform.  
Thus, there is a constant $ C_0$ so large that the set 
\begin{equation}\label{e:W1}
E := \{   \max\{M f ,\ T _\sharp f \}> \tfrac 12 C_0 \langle \lvert  f\rvert  \rangle_{Q_0}\}  
\end{equation}
satisfies $ \mu (E) \le \tfrac 12 \mu (Q_0)$.  
Let $ \mathcal E$ be the collection of maximal elements of $ \mathcal Q$ contained in $ E$.  
We claim that 
\begin{align}\label{e:zclaim}
\lvert  T f (x)\rvert \mathbf 1_{Q_0} & \leq C_0 \langle  \lvert  f\rvert \rangle_{Q_0}  +   
\sum_{P\in \mathcal E} \lvert  T _P f (x)\rvert , 
\\
\textup{where} \quad 
T_P f &:=   \epsilon _{P ^{a}} \langle f \rangle_P \mathbf 1_{P} + \sum_{Q \;:\; Q\subset P} \epsilon _Q \Delta _Q f . 
\end{align}
Here, $ P ^{a}$ is the parent of $ P$.   It is clear that we will add $ Q_0$ to   $ \mathcal S$, the sparse collection defining $ \mathsf S$.  And, one should recurse on the $ T_P f$, at which stage the collection $ \mathcal E$ becomes the $ \mathcal S$-children of $ Q_0$.  Continuing the recursion will complete the proof.  

If $ x\in Q_0 \setminus E$,   certainly    \eqref{e:zclaim} is true.  
For $ x\in E$ then there is a unique $ P\in \mathcal E$ with $ x \in P$, for which we can write 
\begin{align} \label{e:T=}
T f  (x) &=  
 \sum_{Q \;:\;  P ^{a}\subsetneqq  Q } \epsilon _{Q} \Delta _Q f (x)  - \epsilon _{P ^{a}} \langle f \rangle _{ P ^{a}}  + T _P f (x)  . 
\end{align}
Note that the martingale difference $ \Delta _{P ^{a}} f (x)$ is split between the second and third terms above.  
The first two terms on the right in \eqref{e:T=} are, by construction, bounded by $\tfrac 12  C_0 \langle \lvert  f\rvert  \rangle _{Q_0}$.  Thus \eqref{e:zclaim} follows.

\end{proof}
%%%%%%%%%%%%%%%%%%%%%%%%%%%%%% PROOF PROOF PROOF

%%%%%%%%%%%%%%%%%%%%%%%%%%%%%% SECTION  SECTION SECTION
%%%%%%%%%%%%%%%%%%%%%%%%%%%%%% SECTION  SECTION SECTION 
\section{Domination by Sparse Operators:  Euclidean Case} \label{s:euclid}

We extend the martingale proof of \eqref{e:lerner} to the Euclidean setting, yielding an inequality that 
applies to (a) the most general class of Calder\'on-Zygmund operators, and (b)  in the vector valued setting. 

Let $ T$ be an $ L ^2  (\mathbb R ^{d})$ bounded operator, of norm 1, and  with kernel $ K (x,y)$. Namely for 
all $ f, g \in C _0 (\mathbb R ^{d})$, whose  closed supports that  do not intersect, we have 
\begin{equation*}
\langle Tf, g \rangle = \int\!\int K (x,y) f (y) g (x) \; dx\, dy . 
\end{equation*}
Let $ \omega \;:\; [0, \infty ) \mapsto [0, 1 ) $ be a modulus of continuity, that is 
a monotone increasing and subadditive function.  The operator $ T$ is said to be a 
Calder\'on-Zygmund operator if  the kernel $ K (x,y)$ satisfies 
\begin{equation} \label{e:smooth}
\lvert  K (x,y) - K (x',y)\rvert \le \omega \Bigl(\frac {\lvert  x-x'\rvert } {\lvert  x-y\rvert } \Bigr) \frac 1 {\lvert  x-y\rvert ^{d} },  
\end{equation}
and the dual condition, with the roles of $ x$ and $ y$ reversed, holds. 
The classical condition to place on $ \omega $ is the \emph{Dini condition} 
\begin{equation} \label{e:dini}
\int ^{1} _{0}   \omega (t)  \; \frac {dt}t < \infty . 
\end{equation}.

The Theorem below is a stronger version of Lerner's inequality \cite{MR3085756}, and its extension 
in \cite{14094351}*{Corollary A1}. 
Moreover, it only assumes  the Dini condition, while 
prior approaches  \cites{MR3065022,MR3085756,14094351}   require 
$ 1/t$ in the Dini integral be  replaced by $( \log 2/t)/t$.  

%%%%%%%%%%%%%%%%%%%%%%%%%%%%%% THEOREM THEOREM THEOREM
\begin{theorem}\label{t:continuous}
Let $ T$ be a Calder\'on-Zygmund operator for which the modulus of continuity satisfies the Dini condition 
\eqref{e:dini}. 
Then, for any compactly supported function $ f\in L ^{1} (\mathbb R ^{d})$,   we have 
$
\lvert  T f \rvert \lesssim \mathsf S \lvert  f\rvert  
$. 
where $ \mathsf S  = \mathsf S (f, T)$  is  a sum of at most  $ 3 ^{d}$ operators, each of which are 
sparse relative to a choice of dyadic grid on $ \mathbb R ^{d}$.   The same conclusion holds for the maximal 
truncations $ T _{\sharp}$, as defined in \eqref{e:Sharp}. 
\end{theorem}
%%%%%%%%%%%%%%%%%%%%%%%%%%%%%% THEOREM THEOREM THEOREM

To explain the conclusion of the Theorem, we say that $ \mathcal D$ is a  \emph{dyadic grid} of $ \mathbb R ^{d}$,  if $ \mathcal D$
is a collection of cubes $ Q \subset \mathbb R ^{d}$ so that for each integer $ k \in \mathbb Z $, 
the cubes  $ \mathcal D _{k} :=\{Q\in \mathcal D \;:\;  \lvert  Q\rvert = 2 ^{-kd} \}$ partition $ \mathbb R ^{d}$, 
and these collections form an increasing filtration on $ \mathbb R ^{d}$.  
Then, a \emph{sparse operator} is one that is sparse in the sense of \eqref{e:sparse}, relative to a dyadic grid 
and Lebesgue measure.   

\smallskip 

To  prove the Theorem,  we use the  standard fact that the maximal truncations of $ T$ satisfy a weak-$L^1$ inequality, 
see \cite[\S4.2]{MR0290095}. 

%%%%%%%%%%%%%%%%%%%%%%%%%%%%%% LEMMA LEMMA LEMMA
\begin{lemma}\label{l:Tweak} For a Calder\'on-Zygmund operator $ T$ with Dini modulus of continuity, 
we have $ \lambda \lvert  \{ T _{\sharp} f > \lambda \}\rvert \lesssim \lVert f\rVert _1 $, for all $ \lambda >0$, 
where $T _{\sharp} $ are the maximal truncations 
\begin{equation} \label{e:Sharp}
T _{\sharp} f (x) := \sup _{\delta >0}  \Bigl\lvert \int _{\lvert  x-y\rvert > \delta  }  K (x,y)f (y) \; dy\Bigr\rvert . 
\end{equation}
\end{lemma}
%%%%%%%%%%%%%%%%%%%%%%%%%%%%%% LEMMA LEMMA LEMMA

There are a finite  number of choices of grids which well-approximate any cube in $ \mathbb R ^{d}$. 
This old observation has a proof in  \cite[Lemma 2.5]{MR3065022}. 

%%%%%%%%%%%%%%%%%%%%%%%%%%%%%% LEMMA LEMMA LEMMA
\begin{lemma}\label{l:shift} There are choices of dyadic grids $ \mathcal D _{u}$, for $ 1\leq u \leq  3 ^{d}$, 
so that for \emph{any} cube $ P\subset \mathbb R ^{d}$, there is a choice of $ 1\le u \leq 3 ^{d}$, and dyadic 
$ Q\in \mathcal D _u$ such that $  P\subset Q  $,  and $  \ell Q \le 6 \ell P$, where $ \ell P= \lvert  P\rvert ^{1/d} $ is the  side length of $ P$. 
\end{lemma}
%%%%%%%%%%%%%%%%%%%%%%%%%%%%%% LEMMA LEMMA LEMMA

The grids of the Lemma  are defined by appropriate shifts of  a dyadic grid. 
Abusing common terminology, we say that $ Q$ is a \emph{dyadic cube} if it is in any of the collections $ \mathcal D_u$, $ 1\leq u \leq 3 ^{d}$. 
And, we will write $ Q\in \mathcal D _{u (Q)}$.

Since the proof  of Theorem~\ref{t:continuous} is recursive, we need to adapt the definitions of maximal functions and maximal truncations to 
 a cube $ P$.  Set, for $ x\in P$, 
\begin{equation*}
M _{P} f (x) := \sup _{0 < t <  \textup{dist}(x, \partial P)} 
A_t f (x), \qquad  A_t     f  (x) :=   \int \lvert  f (x-y)\rvert   \psi (y/t)\;dy/t ^{d} .
\end{equation*}
Here and below, $ \psi $ is a  smooth function with $ \mathbf 1_{[-1/2, 1/2] ^{d}} \le \psi \le \mathbf 1_{ [-1, 1] ^{d}} $. 
Likewise, the maximal truncations are adapted to a particular cube $ P$: For $ x\in P$, 
 define 
\begin{gather*}
T _{\sharp, P}  f (x) := \sup _{0 <s<  t <  \tfrac 1{12}\textup{dist}(x, \partial P)  }\lvert    T _{s,t} f (x) \rvert, \qquad x\in P,  
\\ \textup{where} \quad 
T_{s, t}  f (x) := 
\int K (x,y) f (y) [ \psi ((x-y)/ t) - \psi ((x-y)/s) ] \; dy . 
\end{gather*}
The truncation levels are taken to stay inside $ P$, as we  measure the distance with respect to the 
$ \ell ^{\infty }$ norm on $ \mathbb R ^{d}$. For $ x\not\in P$, set $M _{P} f (x) = T _{\sharp, P}  f (x)=0$. 
It follows that this definition is monotone, an important fact for us: For $ P\subset Q$
\begin{equation}\label{e:monotone}
T _{\sharp, P} f = T _{\sharp, P} (f \mathbf 1_{P}) \leq  T _{\sharp, Q} f  . 
\end{equation}
Both $ M_P$ and $ T _{\sharp, P}$ still satisfy the weak $ L ^{1}$ inequality.

The main step in the recursion is in this.

%%%%%%%%%%%%%%%%%%%%%%%%%%%%%% LEMMA LEMMA LEMMA
\begin{lemma}\label{l:main} There is a finite constant  $   C  $ so that  
for all integrable functions $ f$  supported on a  cube $   P$,    
there  is a collection $ \mathcal P (P)$ of   dyadic cubes $ Q  \subset  P$, 
so that: 
%%  ENUMERATE
\begin{enumerate}
\item  (The cubes have small measure.)  There holds 
\begin{equation}\label{e:SQ<}
\sum_{Q\in \mathcal P (P)} \lvert  Q \rvert < 3 ^{-3d-3}  \lvert  P\rvert . 
\end{equation}

\item  (The cubes are approximately disjoint.)  If $ Q' \subset Q$, and $ Q',Q\in \mathcal P (P)$,  then $ Q'=Q$. 

\item  ($ Tf$ is controlled.)  Pointwise, there holds 
\begin{equation} \label{e:m2}
 T _{\sharp, P} f    \le C \langle \lvert  f\rvert  \rangle_{P}  +  \max _{ Q\in \mathcal P (P) }   T _{\sharp, Q}  f. 
\end{equation}
\end{enumerate}
%% ENUMERATE
\end{lemma}
%%%%%%%%%%%%%%%%%%%%%%%%%%%%%% LEMMA LEMMA LEMMA

The last point is an extension of \eqref{e:zclaim}. 

%%%%%%%%%%%%%%%%%%%%%%%%%%%%%% PROOF PROOF PROOF
\begin{proof}
For a large constant $ C $ to be determined,  set 
\begin{equation*}
E(C) := \bigl\{ x \in P \;:\;  \max \bigl\{ M _{P} f (x) , T _{\sharp, P} f  (x)\bigr\} \geq C \langle  \lvert  f\rvert \rangle_{P} \bigr\}. 
\end{equation*}
It follows from Lemma~\ref{l:Tweak}, that $ \lvert  E (C)\rvert \lesssim C ^{-1} \lvert  P\rvert  $. 
 For  $ x\in E (C)$,  let  $ 0 <  \sigma _x \leq \textup{dist}( x, \partial P) $ be the largest choice of 
   $0< s <\textup{dist}( x, \partial Q_0) $  such that either $  A_{s}  f (x) = C \langle \lvert  f\rvert  \rangle _{P}$, or 
   there is a      $ s < \tau _x \leq \tfrac 1 {12 }\textup{dist}( x, \partial Q_0) $ with 
$
\lvert   T _{s, \tau _x} f (x)\rvert= C \langle \lvert  f\rvert  \rangle _{P} 
$.

We have this (standard) lower semi-continuity result: There is an absolute $ 0 < \rho < 1$ so that 
\begin{equation} \label{e:zr}
\max \{\lvert   T _{\sigma _x, \tau _x } f (x')\rvert,  A_{\sigma _x}  f (x') \}> \tfrac12 C \langle \lvert  f\rvert  \rangle _{P}, \qquad   \lvert  x-x'\rvert < \rho \sigma _x. 
\end{equation}
For the averaging operator, if  $ A _{\sigma _x } f (x) > C \langle \lvert  f\rvert  \rangle _{P}$, this is very easy, 
and so we proceed under the assumption that $ A _{s } f (x) < C \langle \lvert  f\rvert  \rangle _{P}$ for 
all $ \sigma _x \leq s \leq \textup{dist}( x, \partial Q_0)$.  
It is then   the truncated singular integral that is large, and 
we have to rely upon the Dini condition: Abbreviating a familiar calculation,  and assuming that $ t \lesssim \omega (t)$, for 
$ 0< t < 1$, as we can do, assuming $ \lvert  x-x'\rvert < \rho \sigma _x$, 
\begin{align}
\lvert    T _{\sigma_x,\tau_x} f (x') -   T _{\sigma_x,\tau_x} f (x)\rvert 
& \lesssim 
\int _{  \sigma _x/2< \lvert  x-y\rvert  < \tau _x}  
\omega \Bigl( \frac {\rho \sigma _x} {x-y} \Bigr) \frac{ \lvert  f (y)\rvert } { \lvert  x-y\rvert ^{d}} \; dy 
\\ \label{e:abbrev}
& \lesssim 
\langle \lvert  f\rvert  \rangle_{P} \int _{\sigma _x/2} ^{\infty } 
\omega \Bigl( \frac {\rho \sigma _x} {t} \Bigr) \; \frac {dt} t 
\\
& \lesssim  \langle \lvert  f\rvert  \rangle_{P} \int _0 ^{2\rho} \omega (s) \frac {ds}s.
\end{align}
Thus,  under the Dini condition \eqref{e:dini}, we see that for $\rho = \rho (\psi, \omega ) $ sufficiently small, \eqref{e:zr} holds.  (In particular, $ \rho $ is independent of $ C$.) 

There is a second upper semi-continuity result. Set $ B_x := x +  \sigma _x [-1/2,1/2] ^{d}  \subset P$.  
For any $ y\in B_x$, and any 
$ \frac13\sigma_x < s < t < \tfrac 14\textup{dist}(x, \partial P)$, 
\begin{equation}\label{e:supB} 
\Bigl\lvert \int K (y,z) f (z) [ \psi ({y-z}/{ t}) - \psi ((y-z)/s) ] \; dz \Bigr\rvert \lesssim C \langle \lvert  f\rvert  \rangle_{P}. 
\end{equation}
That is, on $ B_x$, very large values of $ T _{\sharp, P} f (y)$ cannot arise from $ f \mathbf 1_{(2B_x) ^{c}}$, 
since otherwise we violate the definition of $ \sigma _x$. 

\smallskip

We have just seen  that $\rho B_x \subset E (C/2)$. Thus,  
\begin{equation*}
\lvert  E_0 \rvert  \lesssim \rho ^{-d} C ^{-1} \lvert  P\rvert, 
\qquad  E_0:=\bigcup _{x\in E (C)}B_x .  
\end{equation*}
And, if $ y\in P \setminus E_0 $, then $ \lvert  T _{\sharp, P} f (y)\rvert  \leq  C \langle \lvert  f\rvert  \rangle_{P}$. 

The path from the cubes $ B_x$ to dyadic cubes uses Lemma~\ref{l:shift}.  
To each cube $ B_x$, associate a choice of dyadic cube $   Q_x \supset B_x $ with 
 $ \ell Q_x \leq 6 \sigma _x  $.  
Since $ 12\sigma _x < \textup{dist}( x, \partial P)$, it follows that $ Q_x \subset P$.  
Then, we require 
\begin{equation*}
\lvert  E_1 \rvert \leq  3 ^{-3d-3}  \lvert  P\rvert,  \qquad  E_1 := \bigcup _{x\in E_0} Q_x ,  
\end{equation*}
which is true for large enough $ C $, thus specifying this constant.

The collection  $\mathcal P (P)$ claimed in the Lemma are the maximal cubes in 
$
 \{Q_y \;:\; y\in E_0\} 
$. 
Properties (1) and (2) are immediate, and it remains to check (3), namely to bound 
 $ T _{\sharp, P} f (y)$. 
If $ y \in P \setminus E_0$, then $ T _{\sharp, P} f (y)  \leq C \langle \lvert  f\rvert \rangle_P$. 
And, otherwise, $ y \in B _{y} \subset Q$, for some  $Q\in \mathcal P (P)$. 
But it follows from \eqref{e:supB} and monotonicity again   that we have 
\begin{equation*}
 T _{\sharp, P} f (y) \leq C' \langle \lvert  f\rvert  \rangle _{P} + T _{\sharp, B_{y} }f (y) 
\leq C'  \langle \lvert  f\rvert  \rangle _{P} + T _{\sharp, Q_{x} }f (y).
\end{equation*}
Here, $ C'$ is the implicit constant in \eqref{e:supB}.   And, this completes our proof.  
\end{proof}
%%%%%%%%%%%%%%%%%%%%%%%%%%%%%% PROOF PROOF PROOF

%%%%%%%%%%%%%%%%%%%%%%%%%%%%%% PROOF PROOF PROOF
\begin{proof}[Proof of Theorem~\ref{t:continuous}] 
Take a compactly supported function $ f\in L ^{1} (\mathbb R ^{d})$, and  a dyadic cube  $ P \in \mathcal D _{u (P)}$   large enough that $   (T_ {\sharp} f) \mathbf 1_{P} \leq T _{\sharp, P} f + C \langle \lvert  f\rvert  \rangle_P $. 
(The term $ \langle \lvert  f\rvert  \rangle_P$ is added to address boundary effects in the definition of $ T _{\sharp , P} f$.) 
The standard kernel estimate then implies that off of $ P$, there holds 
\begin{equation*}
 T_ {\sharp} f (x)  \mathbf 1_{P ^{c}} (x) \lesssim  \frac { \lVert f\rVert_1} { [ \ell P +  \textup{dist}(x, P) ] ^{d}} . 
\end{equation*}
Apply Lemma~\ref{l:shift} to the cubes $ 2 ^{n}P$, $ n\geq1$, to select dyadic $ P_n$ meeting the conclusion of the Lemma. 
Add $ P $ to $ \mathcal S _ {u (P)}$, and 
$ P_n$ to $ \mathcal S _{u (P_n)}$, where  $ \mathcal S _{u} \subset \mathcal D_u$ is the sparse collection of dyadic
cubes that defines $ \mathsf S _{u}$.   
Then,  it follows  that 
\begin{equation*}
  T _{\sharp} f (x)\lesssim 
\sum_{u=1} ^{3^d}\mathsf S _{u} \lvert  f\rvert (x), \qquad  x \not\in P.     
\end{equation*}

It remains to dominate $ T _{\sharp, P}f$.
Apply Lemma~\ref{l:main}, so that in particular \eqref{e:m2} holds. 
We will rewrite \eqref{e:m2} in the form  below, to set up the recursion. 
\begin{equation*}
   T _{\sharp, P} f  \le C 
   \sum_{u=1} ^{3d}\mathsf S _{u} ^{0} \lvert  f\rvert 
 +  \max _{ Q \in \mathcal P _0 }   T _{\sharp, Q}  (f )
\end{equation*}
Above, all of the sparse operators are zero, except for $ \mathsf S _{u (P)} ^{0}$, and the associated 
sparse collection is $ \mathcal S ^{0} _{u (P)} = \{P\}$. 
We will refer to the cubes in $ \mathcal P (P)= \mathcal P_0$ as \emph{$ \sharp$-descendants of $ P$.}
The inductive step is to apply Lemma~\ref{l:main} to those $ Q \in \mathcal P_0$ with maximal side lengths. 
The details are as follows. Suppose that we have this estimate below for some integer $ t\geq 0$. 
\begin{equation}\label{e:B1}
   T _{\sharp, P} f (x)\le C \sum_{u=1} ^{3d}\mathsf S _{u} ^{t} \lvert  f\rvert 
 +  \max _{  Q\in \mathcal P_t }   T _{\sharp, Q}  (f ), 
\end{equation}
where   these properties hold: 
%%  ENUMERATE
\begin{description}
\item[T0]  $ \mathsf S_{u} ^{t}$ is a  operator defined by  collection  $\mathcal S _{u} ^{t}\subset  \mathcal D_u$, for $ 1\le u \leq 3 ^{d}$, 
as in the equation \eqref{e:Sparse}.  (The operator will be sparse, but we do not assert that at this point.) 

\item[T1] All cubes $ Q\in \mathcal P_t$ are dyadic.   If  $ Q'\subset Q $ for any two $ Q,Q'\in \mathcal P_t$, then $ Q'=Q$.  
Each $ Q\in \mathcal P_t$ is a $ \sharp$-descendant of some cube $ S\in \bigcup _{u=1} ^{3 ^{d}} \mathcal S _{u} $.  

\item[T2] If $ Q \cap S \neq 0$, where  $ Q\in \mathcal P _t$ and $ S \in  \mathcal S ^{t} := \bigcup _{u=1} ^{3 ^{d}}\mathcal S_u ^{t}$, then $ \ell Q < \ell S $. 

\end{description}
%% ENUMERATE

Then, the recursive step is as follows. 
Let $  \mathcal P_ t ^{\ast} $ be those $  Q\in \mathcal P_t $ such that $ \ell Q$ is maximal 
among all side lengths of the cubes $ \{ Q\;:\;  Q\in \mathcal P_t \}$.  
Apply Lemma~\ref{l:main} to each $ Q \in \mathcal P_t ^{\ast }$. We have 
\begin{equation*}
T _{\sharp, Q}  (f ) 
\leq C \langle  \lvert  f\rvert  \rangle _{Q} + 
 \max _{Q'\in  \mathcal P (Q) }  T _{\sharp,  Q'}  (f  ), 
\end{equation*}
where the cubes $ \mathcal P (Q)$ satisfy the conclusions of Lemma~\ref{l:main}. 

It follows from \eqref{e:B1} that pointwise, 
\begin{align*}
   T _{\sharp, P} f & \le C \sum_{u=1} ^{3 ^{d}} \mathsf S _{u} ^{t}  \lvert  f\rvert  
\\& \qquad + \max \Bigl\{ \max _{ Q\in \mathcal P _t \setminus \mathcal P_t ^{\ast}  }   T _{\sharp, Q}  (f  ) ,      \max _{Q \in \mathcal P_t ^{\ast }} \bigl\{ 
C \langle  \lvert  f\rvert  \rangle _{Q} \mathbf 1_{Q}
+  \max _{ Q' \in \mathcal P (Q) }   T _{\sharp, Q'}  (f  ) \bigr\}\Bigr\} 
\\
& \le C  \sum_{u=1} ^{3 ^{d}}\mathsf S _{u}  ^{t+1} \lvert  f\rvert   
+  \max _{Q \in \mathcal P _{t+1} }   T _{\sharp, Q}  (f  ). 
\end{align*}
In the last line, the new  operator $ \mathsf S_u ^{t+1}$ is obtained from $ \mathcal S_u ^{t}$ by 
adding to $ \mathcal S_u ^{t}$ those cubes $ \{ Q  \in \mathcal P_t ^{\ast}  \;:\; u (Q) = u\}$. 
By monotonicity \eqref{e:monotone}, we take the  collection $ \mathcal P_{t+1}$ 
to be  those cubes in 
\begin{equation*}
( \mathcal P _t \setminus \mathcal P_t ^{\ast})   \cup 
\bigcup _{\mathcal P_t ^{\ast}  }  \mathcal P (Q) 
\end{equation*}
that are maximal, with respect to inclusion.  
\smallskip 

After the recursion finishes, it only remains to show that the limiting operators $ \mathsf S _{u}$ are 
sparse. 
Take $ S\in \mathcal S _u$,  and fix the smallest integer $ t$ such that $ S\in \mathcal S_u ^{t}$.  
There are three sources of $ \mathcal S_u$-children for $ S$,  either directly, or through further $ \sharp$-descendants: 
%%  ENUMERATE
\begin{description}
\item[Type A]  
Cubes in $ Q\in \mathcal P (S')$, for some   $S' \in  \mathcal S ^{t} _{u'}$  with $ \ell S' > \ell S$. 

\item[Type B] 
Cubes in $ Q\in \mathcal P (S')$, for some   $S' \in  \mathcal S ^{t}$  with $ \ell S' \leq   \ell S$. 

%\item[Type C] Cubes  $ Q \in \mathcal P_t $, that are not of Type A or   Type B. 

\end{description}
%% ENUMERATE

Cubes of Type A cannot occur:   Since $ \ell S' > \ell S$, the smallest integer $ t'$ such that $ S'\in \mathcal S ^{t'} _{u'}$ 
is strictly smaller than $ t$.  Let $ Q\in \mathcal P (S')$ be contained in $ S$. 
If at any time in the recursion from $ t'\leq s < t$, 
both cubes $ Q$ and $ S$ were potential members of $ \mathcal P _s$, then $ Q$ would have been eliminated by monotonicity, 
namely using \eqref{e:monotone}.

Concerning cubes of Type B,  fix an integer $ n$, and  let $ \mathcal B_n$ be those 
$ S' \in \bigcup _{v=1} ^{3 ^{d}} \mathcal S_v ^{t}$, 
of side length $ 2 ^{-n} \ell S$ that intersect $ S$, but are not contained in it. These are necessarily contained in 
$ (1+ 2 ^{-n} )S \setminus (1- 2 ^{-n} S)$, for dimensional constant $ c$.  And, they can overlap at most $ 3^{d}$ times. 
The total measure of all $ \sharp$-descendants of each $ S'$ is at most $ 3 ^{-3d-3} \lvert  S'\rvert $, 
so the total measure of potential children from  all of these cubes is at most 
\begin{align*}
3 ^{-3d-3}  \sum_{n=0} ^{\infty }\sum_{S'\in \mathcal B_n} \lvert  S'\rvert 
& \leq 3 ^{-2d-3}  \lvert  S\rvert   \sum_{n=0} ^{\infty } (1+2 ^{-n}) ^{d} - (1- 2 ^{-n}) ^{d}  
\\
& \leq 3 ^{-d-3}  \sum_{n=0} ^{\infty } 2 ^{-n+1} \lvert  S\rvert \leq 3 ^{-d-1} \lvert  S\rvert.   
\end{align*}
Therefore, the collections $ \mathcal S_u$ are sparse, and the proof of the Theorem is complete. 

\end{proof}
%%%%%%%%%%%%%%%%%%%%%%%%%%%%%% PROOF PROOF PROOF

%%%%%%%%%%%%%%%%%%%%%%%%%%%%%% SECTION  SECTION SECTION
%%%%%%%%%%%%%%%%%%%%%%%%%%%%%% SECTION  SECTION SECTION 
\section{Complements} \label{s:complements} 

In late 1970's, several authors considered martingale analogs of the $ A_p$ theory. 
For instance, Izumisawa-Kazamaki\cite{MR0436313}, proved a variant of the Muckenhoupt 
maximal function result  \cite{MR0293384 } in this setting. 
When it came to martingale transforms, the distinction between the homogeneous and 
non-homogeneous cases was already recognized by these authors.   
Nevertheless, norm inequalities for martingale transforms were proved by Bonami-L\'epingle \cite[Th. 1]{MR544802} 
in 1978.\footnote{
These references do not seem to be as well known as they should be.  Beyond the setting of discrete martingales 
of this paper, some of these references consider continuous time martingales, with cadalag sample 
paths.} 
Their proof follows the  good-$ \lambda $ approach, which does not yield sharp constants.  

The Lerner mean oscillation  inequality \cite{MR2721744} gives a pointwise bound on an arbitrary 
measurable function $ \phi $ by 
a sum over a sparse dyadic collection of cubes, $ Q\in \mathcal S$. 
Applying this inequality Calder\'on-Zygmund operators $ T$ applied to  $ L ^{1} $ function $ f$, 
 leads a number of `tail issues', see \cites{MR3065022,MR3085756}. 

Lerner-Nazarov and Conde-Rey \cites{14094351,150805639} provide pointwise domination of $ Tf$ by 
sparse operators. Their proofs work on dyadic operators, and `lose a log' in passage to the continuous case.  

The proof just described recursively selects advantageous sparse cubes,  requiring 
 only  the weak-$L^1$ inequality for maximal truncations, and basic facts about dyadic 
grids.  Thus, the technique works with only modest changes for UMD valued functions.  (In contrast, 
the mean oscillation inequality requires  a concept of `median', see \cite{12106236}.)

By an observation of Treil-Vol{\cprime }berg \cite{PC},  the argument in \S~\ref{s:euclid} can be adapted 
to the non-homogeneous Calder\'on-Zygmund theory. In particular, a  
variant of Theorem~\ref{t:continuous} holds in this setting, and it implies this result, 
stated in the language of \cites{MR1881028,MR1470373}.

%%%%%%%%%%%%%%%%%%%%%%%%%%%%%% THEOREM THEOREM THEOREM
\begin{theorem}\label{t:PC}[Treil, Volberg \cite{PC}]  Let $ (X, d)$ be a geometrically doubling space, 
Let $ \mu $ be a measure on $ X$ which is of order $ m$, that is  $ \mu ( B (x,r)) \le r ^{m}  $, 
where $ B (x,r) := \{y \;:\; d (x,y) \le r\}$.  
Let $ T$ be an order $ m$  $ L ^2 $-bounded Calder\'on-Zygmund operator on $ (X, \mu )$.  
For any locally finite weight  $ w$ on $ X$, positive almost everywhere, such that $ \sigma := w ^{-1} $ is 
also a weight, there holds 
\begin{equation*}
\lVert T \;:\; L ^2 (w) \mapsto L ^2 (w)\rVert \lesssim \sup _{x, r} \frac { w (B (x,r))} {\mu (B (x,3r))} \frac { \sigma  (B (x,r))} {\mu (B (x,r))}  . 
\end{equation*}
\end{theorem}
%%%%%%%%%%%%%%%%%%%%%%%%%%%%%% THEOREM THEOREM THEOREM

A third application of the technique of this paper has lead to pointwise bounds for operators beyond the 
Calder\'on-Zygmund scale, an interesting new development by Bernicot-Frey-Petermichl  \cite{151000973}.

Further applications to the multi-linear Calder\'on-Zygmund theory \cite{MR3232584,MR3302105} 
should also be possible. In particular, \cite{14094351}*{Corollary A.1} dominates a multi-linear Calder\'on-Zygmund 
operator by a sparse operator, and from there one deduces weighted inequalities. 
The modulus of continuity of the operator is assumed to satisfy a logarithmic Dini condition, 
which can presumably be relaxed to just a Dini condition.

%%%%%%%%%%%%%%%%%%%%%%%%%%%%%%%%%%%%%%%%%%

\end{document}